\documentclass[10 pt,reqno]{amsart}

\usepackage{amsmath}
\usepackage{amssymb}
\usepackage{amsfonts}
\usepackage{amsthm}
\usepackage{hyperref}
\usepackage{tikz}
\usepackage{enumerate}

\usepackage{bbm}
\usepackage[mathscr]{eucal}

\def\emph#1{{\it #1 }}

\def\inn#1#2{\langle#1,#2\rangle}

\theoremstyle{plain}
\newtheorem{thm}{Theorem}[section]
\newtheorem{lemma}[thm]{Lemma}
\newtheorem{cor}[thm]{Corollary}
\newtheorem{proposition}[thm]{Proposition}

\newtheorem{conjecture}[thm]{Conjecture}
\theoremstyle{remark}

\newtheorem*{remarka}{Remark}

\numberwithin{equation}{section}

\makeatletter
\def\@@and{}
\makeatother

\begin{document}

\title[Fractal local smoothing]
{A fractal local smoothing problem\\
 for the wave equation
}

\author[D. Beltran]{David Beltran}
\address{Departament d'Anàlisi Matemàtica, Universitat de València, Burjassot, Spain}
\email{david.beltran@uv.es}

\author[J. Roos]{Joris Roos}
\address{Department of Mathematics and Statistics, University of Massachusetts Lowell, Lowell, MA, USA}
\email{joris\_roos@uml.edu}

\author[A. Rutar]{Alex Rutar}
\address{Department of Mathematics and Statistics, University of Jyv\"askyl\"a, Finland}
\email{alex@rutar.org}
\author[A. Seeger]{Andreas Seeger}
\address{Department of Mathematics, University of Wisconsin--Madison, Madison, WI, USA}
\email{aseeger@wisc.edu}

\date{January 22, 2025}


\begin{abstract}
For any given set $E\subset [1,2]$, we discuss a fractal frequency-localized version of the $L^p$ local smoothing estimates for the half-wave propagator with times in $E$. A conjecture is formulated in terms of a quantity involving the Assouad spectrum of $E$ and the Legendre transform.
We validate the conjecture for radial functions.
We also prove a similar result for fractal-time $L^2\to L^q$
and square function bounds, for arbitrary $L^2$ functions and general time sets. We formulate a conjecture for $L^p\to L^q$ generalizations.
\end{abstract}

\maketitle
\section{Introduction}

Consider the half-wave propagator
\begin{equation*}
e^{i t \sqrt{-\Delta}}f(x)= \frac{1}{(2\pi)^{d}}
\int_{\mathbb R^d} \widehat{f}(\xi) e^{ i \inn{x}{\xi} + i t |\xi|} {\text{\,\rm d}}\xi,
\qquad x \in \mathbb R^d,\,\, t>0,
\end{equation*}
initially defined for Schwartz functions $f$, where $\widehat f(\xi)=\int_{\mathbb{R}^d} f(y) e^{- i\inn{y}{\xi} } \mathrm{d}y$ denotes the Fourier transform. It is well-known since the work of Miyachi \cite{MiyachiWave} and Peral \cite{Peral1980} that for fixed time $t>0$ and $2\le p<\infty$, there exists a locally bounded constant $C_{t}>0$ such that
\begin{equation}
\label{eq:LS fixed time}
\| e^{i t \sqrt{-\Delta}} f \|_{L^p(\mathbb R^d)} \leq C_{t} \| f \|_{L^p_{s_p} (\mathbb R^d) }, \qquad s_p=(d-1)\big(\tfrac{1}{2}-\tfrac{1}{p}\big).
\end{equation}
Here $L^p_s$ denotes the usual $L^p$-Sobolev space. The result is sharp in the sense that $s_p$ cannot be replaced by a smaller number.
The local smoothing problem for the wave equation, proposed by Sogge \cite{Sogge91}, aims to establish sharp space-time $L^p$-Sobolev estimates for $e^{i t \sqrt{-\Delta}}$ where $t \in [1,2]$. In particular, one aims to gain derivatives over \eqref{eq:LS fixed time} and conjectures that for all $2 < p < \infty$ and all $\varepsilon>0$ there exists a constant $C_{\varepsilon}>0$ such that
\begin{equation}\label{eq:LS conj}
\Big(\int_1^2
\| e^{i t \sqrt{-\Delta}} f \|_{p}^p\, {\text{\,\rm d}} t\Big)^{\frac 1p} \leq C_{\varepsilon} \| f \|_{L^p_{ \sigma_p + \varepsilon}}, \, \,\,\sigma_p = \begin{cases}
0 & \text{ if $2 < p < \frac{2d}{d-1}$}
\\ s_p-\frac 1p & \text{ if $p > \frac{2d}{d-1}$}
\end{cases}.
\end{equation}
The first result of this kind was proved by Wolff \cite{Wolff2000} for large values of $p$.
In two dimensions, Sogge's conjecture
was recently established for all $2 < p < \infty$ by Guth, Wang and Zhang \cite{GWZ}. In \cite{hns} it was also conjectured that for $p>\frac{2d}{d-1}$ the inequality \eqref{eq:LS conj} should hold even with $\varepsilon=0$, and this endpoint result was verified for $p>\frac{2(d-1)}{d-3}$, $d\ge 4$. For $d \geq 3$, the current best result with the $\varepsilon$-loss corresponds to the range $p\geq \frac{2(d+1)}{d-1}$, which was proved by Bourgain and Demeter \cite{BourgainDemeter2015}.
We refer to the survey \cite{BHS-survey} for further history of the problem.
The $L^p$-bound also implies an inequality with $L^p_{\text{\rm rad}}(L^2_{\mathrm{sph}})$ in place of $L^p({\mathbb {R}}^d)$; the version of Sogge's conjecture in this category was proved in \cite{MS-L2radsph}.

In this paper, we introduce a {\it fractal} version of the local smoothing problem, which we validate in the radial case.
To formulate it, we first note that \eqref{eq:LS conj} can be rewritten in a discretized form when frequency localized to an annulus of frequencies $\approx 2^j$ where $j \geq 1$. Define $P_j= \varphi(2^{-j} |D|)$ where $\varphi$ is a smooth bump function supported in the interval $(\frac 14, 4)$. Then a version equivalent to \eqref{eq:LS conj} with $\varepsilon>0$ is that for all $2 < p < \infty$ and $s>\max(s_p, \frac1p)$, there exists a constant $C_s>0$ such that
\begin{equation}\label{eq:LS}
\Big(\sum_{t\in E_j} \|e^{it\sqrt{-\Delta}} P_j
f\|^p_{p}\Big)^{\frac1p} \leq C_s 2^{js}
\|f\|_p
\end{equation}
where
$E_j$ is a
maximal $2^{-j}$-separated subset of $[1,2]$. In the fractal problem we replace $[1,2]$ by an {\it arbitrary} subset $E$, and let $E_j$ be
a {\em $2^{-j}$-discretization of $E$}, i.e., a maximal $2^{-j}$-separated subset of $E$. We then ask how the optimal exponent $s$ is determined by $E$.

Given any bounded $E\subset{\mathbb {R}}$ define the {\it Legendre-Assouad function}
$\nu_E^\sharp\colon {\mathbb {R}}\to {\mathbb {R}}$ by
\begin{equation}\label{eq:nudagger}
\nu_E^\sharp(\alpha) = \limsup_{\delta\to 0} \frac{\log \big( \sup _{\delta\le |I|\le 1} |I|^{-\alpha} N(E\cap I, \delta)\big) }{\log(\frac1\delta)},
\end{equation} where the supremum is taken over all intervals $I$ of length between $\delta$ and $1$.
The terminology in this definition is motivated by Theorem \ref{thm:dualassouad} below.
The quantity $\nu_E^\sharp(\alpha)$ was introduced in a study of circular maximal operators by the first, second and fourth authors in \cite{BRS-fractal}.
It turns out that for {\it all} $E\subset[1,2]$, the critical exponent in \eqref{eq:LS}
can be expressed in terms of $\nu_E^\sharp$, at least in the radial setting.
While $\nu^\sharp_E(\alpha)$ is well-defined for all $\alpha\in\mathbb{R}$, we care about the case $\alpha\ge 0$, because $\nu^\sharp_E(\alpha)=\dim_{\mathrm{M}}E$ for all $\alpha\le 0$, where $\dim_{\mathrm M}E$ denotes the upper Minkowski dimension of $E$.
\begin{thm}\label{thm:LSdiscr}
Let $E \subset [1,2]$ and $2 \leq p < \infty$. Then for every $\varepsilon>0$ there exist a constant $C_{\varepsilon,p} >0$ such that for all $j\ge 1$ and all $2^{-j}$-discretizations $E_j$ of $E$,
\begin{equation}\label{eq:LS thm}
\Big( \sum_{t\in E_j} \|e^{i t \sqrt{-\Delta}} P_j f\|_p^p \Big)^{1/p}\le C_{\varepsilon,p} 2^{j (\frac{1}{p}\nu^\sharp_E(ps_p)+\varepsilon)} \| f \|_{L^p_{\mathrm{rad}}}
\end{equation}
for all radial $L^p$ functions $f$. Moreover, the inequality is sharp up to the $\varepsilon$-loss.
\end{thm}
The proof of the upper bound is fairly standard: it is a refinement of the argument in \cite{MS-radial} for $E=[1,2]$ (see also
\cite{ColzaniCominardiStempak})
and extends it as an essentially sharp result for arbitrary sets $E\subset [1,2]$. Note that for $E=[1,2]$ we have
\begin{equation}\label{eq:LS nu [1,2]}
\tfrac1p \nu^\sharp_{[1,2]}( ps_p) = \left\{
\begin{array}{ll}
\frac1p & \text{if }\;p \leq \frac{2d}{d-1}\\
s_p & \text{if }\; p>\frac{2d}{d-1}
\end{array},
\right.
\end{equation}
matching the exponents in the standard local smoothing conjecture \eqref{eq:LS}.
It is reasonable to conjecture that Theorem \ref{thm:LSdiscr} holds for all $L^p$ functions; this constitutes a fractal analogue of Sogge's conjecture for general $L^p$ functions.

The Legendre-Assouad function
is closely related to the
\textit{Assouad spectrum} of $E$, which we now recall.
For $0\le \theta< 1$ define $\dim_{\mathrm A, \theta} E$ as the infimum over all exponents $a>0$ for which there exists a constant $C$ such that
\[N(E\cap I,\delta)\le C (|I|/\delta)^{a}\]
for all intervals $I$ with $|I|=\delta^\theta$ and $\delta\in (0,1)$
(Fraser--Yu \cite{FraserYu2018}). The {\em Assouad spectrum} is the function $\theta\mapsto \dim_{{\mathrm A}, \theta} E$.
At $\theta=0$ we recover the upper Minkowski dimension
\[\dim_{\mathrm M} E=\dim_{\mathrm A,0} E.\]
The Assouad spectrum is continuous on $[0,1)$ and the limit
\[\dim_{\mathrm qA} E= \lim_{\theta\to 1-} \dim_{{\mathrm A}, \theta} E\] exists and is called the {\it quasi-Assouad dimension} (L\"u--Xi \cite{L"uXi2016}).
We refer the reader to Fraser's monograph \cite{FraserBook} for further information.

The {\em Legendre transform} of a (not necessarily convex) continuous function $\nu$ defined on a closed interval $I\subset\mathbb{R}$ is defined by
\begin{equation}\label{eq:legendre}
\nu^*(\alpha) = \sup_{\theta\in I} \theta\alpha - \nu(\theta),
\end{equation}
which is finite for all $\alpha\in\mathbb{R}$ if $I$ is compact.
It was observed in \cite{BRS-fractal} that $\nu^\sharp_E$ equals the Legendre transform of
\begin{equation}\label{eq:nudef}
\nu_E(\theta) = -(1-\theta) \dim_{{\mathrm A}, \theta} E,\quad \theta\in [0,1],
\end{equation}
under a certain regularity assumption on $E$.
The function $\nu_E$ is increasing, but may not be convex (see \cite{FraserYu2018, Rutar24}). The regularity assumption can be removed, and combining this with the characterization of
Assouad spectra by the third author \cite{Rutar24} allows us to obtain a simple characterization of the class of functions which occur as Legendre--Assouad function of
some subset of $[1,2]$. \footnote{The value of $\nu_E^\sharp(\alpha)$ does not change under dilations and translations, so if $\tau=\nu_E^\sharp $ for some bounded $E$, then also $\tau=\nu_{E'}^\sharp$ for some $E'\subset[1,2]$.}
\begin{thm} \label{thm:dualassouad} The following hold:

(i) For all bounded $E\subset {\mathbb {R}}$, $\nu^\sharp_E=\nu^*_E$.

(ii) A function $\tau:[0,\infty)\to [0,\infty)$ satisfies $\nu_E^\sharp|_{[0,\infty)}=\tau$
for some bounded set $E\subset {\mathbb {R}}$
if and only if $\tau$ is increasing, convex, and satisfies $\tau(\alpha)=\alpha$ for $\alpha\ge 1$.
\end{thm}
As a consequence of (i), $\nu^\sharp_E$ only depends on the convex hull of $\nu_E$ which by convex duality is equal to $(\nu^\sharp_E)^*$. This together with the characterization of increasing Assouad spectra in \cite[Cor. B]{Rutar24} gives (ii). We provide the details in \S \ref{sec:Assouad}.

\begin{cor}\label{cor:nusharp}
Let $E$ be bounded and $\gamma=\dim_{\mathrm{qA}} E$. Then
\begin{equation}\label{eq:nusharpprop}
\nu_E^\sharp(\alpha) = \alpha\quad \text{if}\quad \alpha\ge \gamma
\end{equation}
and the number $\gamma$ is minimal with this property.
\end{cor}
Note that \eqref{eq:nusharpprop} was already observed in \cite{BRS-fractal}.
For $0\le\alpha\le \gamma$, $\nu_E^\sharp(\alpha)$ can be interpreted as a new dimensional spectrum interpolating between Minkowski and quasi-Assouad dimension. If $\beta<\gamma$, then $\nu_E^\sharp(\alpha)$ is strictly increasing for $\alpha\ge 0$.

Applying \eqref{eq:nusharpprop} to the sharp exponent in \eqref{eq:LS thm} we obtain
$ \tfrac1p\nu^\sharp_E(ps_p) = s_p$
for $p\ge p_\gamma=\tfrac{2(d-1+\gamma)}{d-1},$
where $\gamma=\dim_{\mathrm{qA}}E$.
This implies that if the standard local smoothing conjecture \eqref{eq:LS} is known for some $p_\circ \geq \frac{2d}{d-1}$, then the corresponding fractal conjecture is also true for all $p \geq p_\circ$ and all $E \subset [1,2]$. However for $p < \frac{2d}{d-1}$, the fractal problem differs from the classical one, and for general $E\subset [1,2]$
the supercritical regime $p \geq \frac{2d}{d-1}$ in \eqref{eq:LS nu [1,2]} is replaced by $p \geq \frac{2(d-1+\gamma)}{d-1}$.
In particular, we have the following.
\begin{cor}\label{cor:radial}
For every $E\subset [1,2]$ with $\gamma=\dim_{\mathrm{qA}}E$, $p \geq p_\gamma= \frac{2(d-1+\gamma)}{d-1}$, $\varepsilon>0$, and radial $f$,
\begin{equation*}
\Big(\sum_{t\in E_j} \|e^{i t \sqrt{-\Delta}} P_j f\|_p ^p \Big)^{1/p}\, \leq
C_{\varepsilon,p} \, 2^{j( s_p+ \varepsilon)}
\|f\|_{L^p_{\text{\rm rad}}}.
\end{equation*} The exponent is sharp up to the $\varepsilon$-loss.
\end{cor}

Theorem \ref{thm:dualassouad} illustrates a striking contrast to the classical local smoothing problem: solving the fractal smoothing problem for $p=p_\gamma$ does typically {\em not} imply sharp estimates in the range $2 < p < p_\gamma$ by interpolating with $p=2$.
Indeed, for this interpolation to be sharp it is necessary that
\begin{equation}\label{eq:interpolated}
\tfrac 1p\nu_E^\sharp (ps_p)= \begin{cases} \big(1-\tfrac{\beta}{\gamma}\big)s_p + \tfrac{\beta}{p} \quad &\text{ if } \, 2 \leq p \leq p_\gamma
\\ s_p,
\quad &\text{ if } \, p\ge p_\gamma.
\end{cases}
\end{equation}
That is, $\nu^\sharp_E$ consists of two affine linear pieces. But Theorem \ref{thm:dualassouad} (ii) says in particular that the function $\nu^\sharp_E$ need not be piecewise affine. This is the same phenomenon observed in \cite{RoosSeeger} for the $L^p\to L^q$ type sets of spherical maximal functions (although $\nu_E^\sharp$ was not mentioned explicitly there).

The interpolated exponents \eqref{eq:interpolated} occur
if
$E$ is {\em quasi-Assouad regular}, that is if its (upper) Assouad spectrum takes the form
\begin{equation}\label{eqn:Assouadregular} \dim_{\mathrm A,\theta} \!E= \begin{cases} \frac{\beta}{1-\theta} &\text{ if } 0<\theta\le 1-\frac \beta\gamma,\\ \gamma &\text{ if }1-\frac \beta\gamma\le \theta<1,
\end{cases}
\end{equation}
and $\dim_{\mathrm{A},0} E=\beta= \dim_{\mathrm{M}} E$ and $\dim_{\mathrm{A},1} E=\gamma=\dim_{\mathrm{qA}} E$.
The equation can be interpreted as saying that the Assouad spectrum should always achieve its largest possible value when given the endpoint values $\beta$ and $\gamma$ (indeed, $\dim_{\mathrm A,\theta}E$ is always bounded by the right hand-side in \eqref{eqn:Assouadregular}, see \cite{FraserYu2018}).
Examples include
all cases where $\beta=\gamma$ (such as self-similar Cantor-type sets), and convex sequences $E=\{1+n^{-a}\,:\,n\ge 1\}$ with $a>0$, where $\beta=(a+1)^{-1}$ and $\gamma=1$. For other examples see \cite{RoosSeeger,Rutar24}.
The simplest examples where \eqref{eq:interpolated} fails are of the form $E=E_1\cup \dots\cup E_k$ with $E_1,\dots, E_k$ quasi-Assouad regular.
Then
\[\nu^\sharp_E=\max(\nu^\sharp_{E_1}, \dots, \nu^\sharp_{E_k}).\]
\begin{remarka}
The above definition of quasi-Assouad regular sets is equivalent to that introduced in \cite{RoosSeeger} (see \cite[Cor. C]{Rutar24} for the equivalence).
\end{remarka}

The Legendre-Assouad function
is also relevant for other related estimates with a fractal feature (e.g. circular maximal functions \cite{BRS-fractal}).
As a further example we prove certain Strichartz type
estimates
for $e^{it\sqrt{-\Delta}}$ with fractal sets of times $E \subset [1,2]$. In this case we obtain a result valid for all $L^2$-functions (not necessarily radial).

\begin{thm} \label{thm:sqfct}
Let $E \subset [1,2]$, $2\le r\le q<\infty$ and \[ s> s_{E}(q) =
\tfrac{d+1} 2 (\tfrac 12-\tfrac 1 q) +
\tfrac 1q \nu_E^\sharp\big(\tfrac{d-1}{2}(\tfrac q2-1)\big).\] Then there exists a constant $C_{s,q}>0$ such that for all $f\in L^2$
\begin{equation}\label{eq:Lqell2}
\Big\| \Big(\sum_{t\in E_j} |e^{-it\sqrt{-\Delta}} P_j f|^r\Big)^{1/r} \Big\|_q \leq C_{s,q} 2^{j s} \|f\|_2.
\end{equation}
{Moreover, if $s < s_{E}(q) $
this conclusion fails to hold.}
\end{thm}

The critical exponent $s_{E}(q)$ does not depend on $r$. Thus the upper bound follows from the case $r=2$. The corresponding square function is relevant to variation bounds for spherical averages: for recent results and further references see \cite{BeltranOberlinRoncalSeegerStovall}, \cite{Wheeler2024}.
The proof of the upper bounds in Theorem \ref{thm:sqfct} is based on familiar $TT^*$ arguments; it can be seen as a refinement of a result in \cite{AHRS} for $r=q$.
We note that it
is mainly interesting for us in the range $2\le q<q_\gamma= \frac{2(d-1+2\gamma)}{d-1}$, with $\gamma=\dim_{\mathrm{qA}}E$. Since $\nu^\sharp_E(\alpha)=\alpha$ for $\alpha\ge \gamma$ and since $\frac{d-1}{2}(\frac q2-1)\ge \gamma$ if and only if $q\ge q_\gamma$, we see that the operator norm in \eqref{eq:Lqell2} is $\lesssim_{q,\varepsilon} 2^{jd(\frac 12-\frac 1q)+j\varepsilon} $ for $q\ge q_\gamma$. If one replaces the quasi-Assouad dimension $\gamma$ by the Assouad dimension
a stronger result for $q\ge q_\gamma$ can be proven with $\varepsilon=0$:
see Proposition
\ref{prop:Str-Assouad} below.

The case $r=q$ in Theorem \ref{thm:sqfct} is a specific case of a more general $L^p\to L^q$ fractal local smoothing conjecture.
\begin{conjecture} \label{conj:LpLq}
Let $E \subset [1,2]$ and $1 < p \leq q < \infty$, $q > p'$. Then for every $s> s_E(p,q):=\tfrac{d+1}{2}(\tfrac{1}{p}-\tfrac{1}{q}) + \tfrac{1}{q}\nu_E^\sharp \big( \tfrac{q(d-1)}{2}(1-\tfrac{1}{p} - \tfrac{1}{q}) \big)$, there exists a constant $C_{s,p,q} >0$ such that
\[\Big(\sum_{t\in E_j} \|e^{it\sqrt{-\Delta}} P_j
f\|^q_q\Big)^{\frac1q} \leq C_{s,p,q} 2^{js}
\|f\|_p
\]
holds for all $j\ge 1$ and all $2^{-j}$-discretizations $E_j$ of $E$.
\end{conjecture}

It will be shown in Proposition \ref{prop:LpLqlower} that $s_E(p,q)$ cannot be replaced by a smaller value.
We note that $s_{E}(2,q)=s_{E}(q)$ in Theorem \ref{thm:sqfct}, and thus the conjecture
is verified for $p=2$. For $q=p$, we note that $s_{E}(p,p)=\frac{1}{p} \nu_E^\sharp (p s_p)$, which matches the exponent in Theorem \ref{thm:LSdiscr} and the corresponding conjecture for general $L^p$-functions. For $E=[1,2]$, one recovers the numerology in \cite[Conjecture 1.1]{Beltran-Saari} and for $E=\{t_0\} \subset [1,2]$, it coincides with the numerology of the fixed-time $L^p-L^q$ estimate, which follows from interpolating \eqref{eq:LS fixed time} with the standard bound $\|e^{i t \sqrt{-\Delta}} P_j\|_{L^1\to L^\infty} \lesssim 2^{j\frac{d+1}{2}}$.
Note that $s_{E}(p,q)= s_q + \frac{1}{p}-\frac{1}{q}$ for
$q\ge p' \frac{d-1+2\gamma}{d-1}$
and all $E \subseteq [1,2]$. Under this condition, the case $q=p$ of the above conjecture implies the $q > p$ case by interpolation with the $L^1\to L^\infty$ bound, similarly to the case $E=[1,2]$. However, this is not generally true if
$q< p' \frac{d-1+2\gamma}{d-1}$ and $E \subset [1,2]$ is arbitrary.

\begin{remarka}
Different types of fractal
space-time problems for the wave equation or spherical means have been considered in the literature, see for example
\cite{Oberlin06, ChoHamLee, IosevichKrauseSawyerTaylorUriarte, HamKoLee-fractal,Wheeler2024}.
These authors put a fractal measure $\mu$ on space-time $\mathbb{R}^d\times \mathbb{R}$ and ask for corresponding $L^p({\mathbb {R}}^d)\to L^q(\mu)$ estimates. The currently known results seem to be complementary to ours, in that they are concerned with regimes of exponents in which the various notions of dimension seem to make no difference. See also Wheeler \cite{Wheeler2024} where Corollary \ref{cor:radial} is conjectured for all $f \in L^p$ in the case when the Minkowski and Assouad dimensions of $E$ coincide. Many interesting questions arise.

\end{remarka}

\subsubsection*{Notational conventions}

Given a list of objects $L$ and real numbers $A$, $B \geq 0$, here and throughout we write $A \lesssim_L B$ or $B \gtrsim_L A$ to indicate $A \leq C_L B$ for some constant $C_L$ which depends on only items in the list $L$. We write $A \sim_L B$ to indicate $A \lesssim_L B$ and $B \lesssim_L A$. We say that a real-valued function $f$ on the real line is increasing (as opposed to non-decreasing) if $f(x)\le f(y)$ whenever $x\le y$ are in its domain.

\subsubsection*{Structure of the paper} In \S\ref{sec:Assouad} we discuss properties of $\nu_E^\sharp$ and prove Theorem \ref{thm:dualassouad}. In \S\ref{sec:lowerboundsrad} we prove sharpness of Theorem \ref{thm:LSdiscr}. In \S\ref{sec:lowerboundsnonradial} we prove sharpness of Theorem \ref{thm:sqfct} and motivate the numerology in the $L^p\to \ell^q_{E_j}(L^q)$ conjecture.
In \S\ref{sec:LS} we prove the upper bounds in Theorem \ref{thm:LSdiscr} and in \S\ref{sec:SF} we prove the upper bounds in Theorem \ref{thm:sqfct}.

\subsubsection*{Acknowledgements}

D.B., J.R. and A.S. were supported through the program {Oberwolfach Research Fellows} by Mathematisches Forschungsinstitut Oberwolfach in 2023. D.B. was supported in part by the AEI grants RYC2020-029151-I and PID2022-140977NA-I00. J.R was supported in part by NSF grant
DMS-2154835. J.R. also thanks the Hausdorff Research Institute for Mathematics in Bonn for providing a pleasant working environment during the Fall 2024 trimester program.
A.R. was supported by Tuomas Orponen's grant from the Research Council of Finland via the project Approximate Incidence Geometry, grant no. 355453. A.S. was supported in part by NSF grant DMS-2348797.

\section{The Legendre--Assouad function}\label{sec:Assouad}
In this section we prove Theorem \ref{thm:dualassouad} and Corollary \ref{cor:nusharp}.
Let us first recall some basic facts about the Legendre transform.
Let $\nu\colon I\to\mathbb{R}$ be a continuous function that is not necessarily convex, defined on a closed interval $I\subset\mathbb{R}$.
Its Legendre transform $\nu^*$ (defined by \eqref{eq:legendre}) is always convex, as a supremum of affine functions, and a basic fact is convex duality: the function $\nu^{**}=(\nu^*)^*$ is the convex hull of $\nu$, i.e.\ it is the largest convex function bounded above by $\nu$. In particular, we have $\nu=\nu^{**}$ if and only if $\nu$ is convex.
Note that here we adopt the convention to extend $\nu$ to a function on $\mathbb{R}$ by setting $\nu(\theta)=\infty$ for $\theta\not\in I$, so that both $\nu$ and $\nu^*$ are defined on all of $\mathbb{R}$.
More details can be found in e.g.\ \cite{Rockafellar}.

\begin{proof}[Proof of Theorem \ref{thm:dualassouad} (i)]
Let $\gamma_E(\theta) = \dim_{\mathrm{A},\theta} E$ and
\begin{equation*}
\varphi(\delta,\theta)= \frac{\sup_{|I|=\delta^\theta}{ \log N(E\cap I,\delta)}}{(1-\theta)\log(\frac 1\delta)}.
\end{equation*}
Then the claim can be rewritten as
\begin{equation}\label{eq:nusharplegendre1}
\limsup_{\delta\to 0} \sup_{\theta\in[0,1]} \theta\alpha + (1-\theta)\varphi(\delta,\theta) = \max_{\theta\in [0,1]} \theta \alpha + (1-\theta)\gamma_E(\theta).
\end{equation}

Fix $\alpha\ge 0$.
We first prove the lower bound in \eqref{eq:nusharplegendre1}.
Since the Assouad spectrum is continuous, there exists $\theta_\alpha\in [0,1]$ such that
the right-hand side of \eqref{eq:nusharplegendre1} equals
\begin{equation}\label{eq:assouad1pf2} \theta_\alpha \alpha + (1-\theta_\alpha)\gamma_E(\theta_\alpha).\end{equation}
By taking $\theta=\theta_\alpha$ in the supremum, \[ \limsup_{\delta\to 0} \sup_{\theta\in[0,1]} \theta\alpha + (1-\theta)\varphi(\delta,\theta)\ge \theta_\alpha\alpha + (1-\theta_\alpha)\limsup_{\delta\to 0}\varphi(\delta,\theta_\alpha),\]
which is equal to \eqref{eq:assouad1pf2}, concluding the proof of the lower bound.

To prove the upper bound let $\varepsilon>0$. Let $(\delta_n)$ be a monotone sequence converging to zero so that the left-hand side of \eqref{eq:nusharplegendre1} equals
\[ \lim_{n\to\infty} \sup_{\theta\in[0,1]} \theta\alpha + (1-\theta)\varphi(\delta_n,\theta). \]
By definition of the supremum, for every $n\in\mathbb{N}$ there exists $\theta_n\in [0,1]$ such that
\begin{equation} \label{eq:zerotheps}\sup_{\theta\in[0,1]} \theta\alpha + (1-\theta)\varphi(\delta_n,\theta)\le \theta_n\alpha + (1-\theta_n)\varphi(\delta_n,\theta_n) + \varepsilon. \end{equation}
By passing to a subsequence we may assume that $(\theta_n)$ converges to a limit $\theta_*\in [0,1]$.
By continuity of $\theta\mapsto (1-\theta)\gamma_E(\theta)$ at $\theta_*$ we may choose a value $\theta_*^-=\theta_*^-(\varepsilon)$ in $[0,\theta_{*})$ close to $\theta_*$ so that

\begin{equation}\label{eq:firsteps}
(1-\theta_*^-)\gamma_E(\theta_*^-) \le (1-\theta_*) \gamma_E(\theta_*)+\varepsilon
\end{equation}

Further since $\theta_n\to \theta_*$ there exists $N_\varepsilon$ so that for all $n\ge N_\varepsilon$ we have
\begin{equation}\label{eq:secondeps}
\theta_n \ge \theta_*^-\quad\text{and}\quad \theta_n\alpha \le \theta_*\alpha + \varepsilon.\end{equation}
Then there exists a constant $C_\varepsilon>0$
such that
\begin{equation}\label{eq:thirdeps} \sup_{|I|=\delta_n^{\theta_n}} N(E\cap I,\delta_n)\le \sup_{|I|=\delta_n^{\theta_*^-}} N(E\cap I,\delta_n) \le C_{\varepsilon} \delta_n^{-(1-\theta_*^-)\gamma_E(\theta_*^-)-\varepsilon} \end{equation}
for $n\ge N_\varepsilon$ (using $\theta_n\ge \theta_*^-$ in the first inequality and the definition of Assouad spectrum in the second).
Hence, for $n\ge N_\varepsilon$,
\[(1-\theta_n)\varphi(\delta_n,\theta_n) \le \tfrac{\log C_\varepsilon}{\log (\frac{1}{\delta_n})} + (1-\theta_*^-)\gamma_E(\theta_*^-)+\varepsilon.\]
By making $N_\varepsilon$ larger if needed we may also assume $\log(C_\varepsilon)/\log(1/\delta_n)\le \varepsilon$ for all $n\ge N_\varepsilon$.
Then from \eqref{eq:firsteps} and \eqref{eq:thirdeps}
\[ (1-\theta_n)\varphi(\delta_n,\theta_n) \le (1-\theta_*^-)\gamma_E(\theta_*^-) + 2\varepsilon \le (1-\theta_*)\gamma_E(\theta_*) + 3\varepsilon \]
for $n\ge N_\varepsilon$.
Combining this with
\eqref{eq:zerotheps} and \eqref{eq:secondeps} we conclude
\begin{align*} \sup_{\theta\in[0,1]} \theta\alpha + (1-\theta)\varphi(\delta_n,\theta) & \le \theta_* \alpha + (1-\theta_*)\gamma_E(\theta_*) + 5\varepsilon \\
&\le \max_{\theta\in [0,1]} \theta\alpha + (1-\theta)\gamma_E(\theta) + 5\varepsilon.
\end{align*}
Since $\varepsilon>0$ was arbitrary this concludes the proof.
\end{proof}

\begin{proof}[Proof of Theorem \ref{thm:dualassouad} (ii)]
By part (i), $\nu^\sharp_E$ is convex, because it equals the Legendre transform of $\nu_E$. It is increasing since $\theta\ge 0$ in the maximum in \eqref{eq:nusharplegendre1}. Finally, if $\alpha\ge 1$, then $\nu^\sharp_E(\alpha)= \alpha$:
the lower bound always holds by taking $\theta=1$ and the upper bound follows from $\dim_{\mathrm{A},\theta} E\le 1$ and $\alpha\ge 1$.

To show the converse, let $\tau\colon [0,\infty)\to [0,\infty)$ be increasing, convex and $\tau(\alpha)=\alpha$ for $\alpha\ge 1$.
Note that $\tau(\alpha)\ge \alpha$ for all $\alpha\in [0,\infty)$ by convexity and assumption.
Define the function $\nu\colon [0,1]\to \mathbb{R}$ by
\[ \nu(\theta) = \tau^*(\theta) = \sup_{\alpha\ge 0} \alpha\theta - \tau(\alpha) \] for $0\le \theta\le 1$.
(Note $\tau^*(\theta)$ is defined for all $\theta\le 1$.) Note that $\nu(1)=0$ since $\alpha \leq \tau(\alpha)$ and $\tau(1)=1$.
For $\theta\in (0,1)$ define
\begin{equation}\label{eq:om-def}\gamma(\theta)=\frac{-\nu(\theta)}{1-\theta}.\end{equation}

We now use the characterization of the class of increasing functions that are attainable as the Assouad spectrum of a bounded set $E\subset {\mathbb {R}}$ from \cite[Cor. B]{Rutar24}. It states that if $\gamma\colon (0,1)\to [0,1]$ is an increasing function such that \[\theta\mapsto \nu(\theta)=-(1-\theta)\gamma(\theta)\] is increasing on $(0,1)$, then $\gamma$ is the Assouad spectrum of a bounded subset $E\subset{\mathbb {R}}$.

Let us verify these assumptions for $\gamma$ as defined in \eqref{eq:om-def}. First, $\gamma$ is increasing because $\nu=\tau^*$ is convex; indeed, since $\nu(1)=0$, we have $\nu(\theta t + (1-t)) \le t \nu(\theta)$ for all $t,\theta\in [0,1]$.
Second,
$\nu$ is increasing because
the supremum in its definition is taken over $\alpha\ge 0$. Finally, $\gamma$ takes values in the interval $[0,1]$: the inequality $\gamma(\theta)\ge 0$ follows because $\nu(1)=0$ and $\nu$ is increasing, and the inequality $\gamma(\theta)\le 1$ is equivalent to $\tau^*(\theta)\ge \theta-1$ which holds because $\tau(1)=1$ (take $\alpha=1$ in the supremum defining $\tau^*$).
Thus, we obtain a bounded $E$
with $\nu_E=\nu$, where $\nu_E(\theta)=-(1-\theta)\dim_{\mathrm A, \theta}E$ as in \eqref{eq:nudef}. By part (i) of Theorem \ref{thm:dualassouad} and convex duality,
\[ \nu^\sharp_E = \nu_E^* = \nu^* = \tau^{**} = \tau \]
which concludes the proof.
\end{proof}

\begin{proof}[Proof of Corollary \ref{cor:nusharp}]
Let $\gamma_E(\theta)=\dim_{\mathrm{A},\theta} E$, $\gamma=\dim_{\mathrm{qA}}E$.
By Theorem \ref{thm:dualassouad} (i),
\begin{equation}\label{eq:nusharpcorpf} \nu^\sharp_E(\alpha) = \sup_{\theta\in [0,1]} \alpha \theta + (1-\theta)\gamma_E(\theta).
\end{equation}
First note that $\nu^\sharp_E(\alpha)\ge \alpha$ for all $\alpha\in\mathbb{R}$ by taking $\theta=1$ in the supremum.
If $\alpha\ge \gamma$, then $\nu^\sharp_E(\alpha)\le \alpha$ holds also by using
$\gamma_E(\theta)\le \gamma$.
To show the minimality claim suppose $\nu^\sharp_E(\alpha)>\alpha$. It then suffices to show that $\alpha<\gamma$.
By continuity of the Assouad spectrum, for every $\alpha\in\mathbb{R}$ there exists $\theta_\alpha\in [0,1]$ where the supremum in \eqref{eq:nusharpcorpf} is attained.
Since $\gamma_E(\theta_\alpha)\le \gamma$ we obtain $\alpha < \nu^\sharp_E(\alpha) \leq \alpha \theta_\alpha + (1-\theta_\alpha)\gamma$, equivalently $(1-\theta_\alpha)\alpha < (1-\theta_\alpha)\gamma$ and $\theta_\alpha\not=1$. Thus $\alpha<\gamma$, as required.
\end{proof}

\section{Lower bounds in Theorem \ref{thm:LSdiscr} }\label{sec:lowerboundsrad} In this section we test the half-wave operator on suitable radial functions to show the (essential) sharpness of Theorem \ref{thm:LSdiscr}.
Let $I \subset [1,2]$ be an interval of length $|I| \geq M 2^{-j}$ containing points in $E_j$, with $M\geq 1$ a sufficiently large constant chosen below.
It suffices to prove that
\begin{equation}\label{eq:lower bound 1}
\sup_{\| g \|_p \leq 1} \,\, \sum_{t \in E_j \cap I} \|e^{i t \sqrt{-\Delta}} g \|_{p}^p \gtrsim N(E \cap I, 2^{-j}) |I|^{- p s_p}.
\end{equation}
By the definition of $\nu_E^\sharp$ in \eqref{eq:nudagger}, this implies that \eqref{eq:LS thm} is sharp up to the $\varepsilon$-loss.

To this end, let $I'$ be a subinterval of $I$ with length $|I|/2$ such that $N(E\cap I',\delta) \ge \frac 12 N(E\cap I,\delta)$ and let $t_I$ be the boundary point of $I$ such that ${\text{\rm dist}}(t_I, I')\ge |I|/4$; without loss of generality, we assume that $t_I$ is the left endpoint.

Consider the radial function $g_I$
given by \[\widehat{g_I}(\xi)= e^{-i t_I |\xi|} \varphi(2^{-j}|\xi|)\] where $\varphi$ is a nonnegative bump function on $(1/2,2)$. We first observe \begin{equation}\label{eq:gJLp}\|g_I\|_p \lesssim 2^{j(\frac{d+1}{2}-\frac 1p)}. \end{equation}
Indeed from Plancherel's theorem $\|g_I\|_2\le 2^{jd/2}.$ For an $L^\infty$ bound we observe first that a multiple integration-by-parts yields $|g_I(x)|= O(1)$ for $|x|<1/2$ and $|x|\ge 3$. For $1/2<|x|<3$ we use the Fourier inversion formula for radial functions \cite[\S IV.3]{stein-weiss}
\[ g_I(x) = (2\pi)^{-d/2}
\int_0^\infty e^{- i t_I s} \varphi(2^{-j}s)
J_{\frac{d-2}{2}}(s|x| ) (s|x|)^{-\frac{d-2}{2} }
s^{d-1} {\text{\,\rm d}} s.\]
Recall also from \cite[\S IV.3]{stein-weiss} the well-known asymptotics for $|u|\ge 1$,
\begin{equation}\label{eq:Besselasympt} J_{\frac{d-2}{2}} ( u) =\big ( e^{-i ( u - \frac \pi 4(d-1))} + e^{i ( u - \frac \pi 4(d-1))} \big) (2\pi u)^{-1/2} + R(u)
\end{equation}
where $|R(u)|=O(|u|^{-3/2})$ for $|u|\ge 1$. From this we see that $|g_I(x)| \lesssim 2^{j(d+1)/2}$, which is then also the bound for $\|g_I\|_\infty$. Thus \eqref{eq:gJLp} follows using $\|g\|_p\le \|g\|_2^{2/p} \|g\|_\infty^{1-2/p} $ for $2\le p\le \infty$. We note that by a slightly more careful argument one can show the pointwise bound $|g_I(x)|\lesssim_N 2^{j\frac{d+1}{2}}(1+2^j||x|-t_I|)^{-N}$ which also gives \eqref{eq:gJLp}.

We now turn to lower bounds for $\| e^{i t \sqrt{-\Delta}} g_I \|_p$. Again by the Fourier inversion formula for radial functions we have
\begin{equation*}
e^{i t \sqrt{-\Delta}}g_I(x) = (2\pi)^{-d/2}
\int_0^\infty e^{i (t-t_I)s} \varphi(2^{-j}s)
J_{\frac{d-2}{2}}(s|x| ) (s|x|)^{-\frac{d-2}{2} }
s^{d-1} {\text{\,\rm d}} s.
\end{equation*}
Then, for $|x| \geq 2^{-j+2}$ we can write using \eqref{eq:Besselasympt}
\begin{equation}\label{eq:lower bound 2}
e^{i t \sqrt{-\Delta}} g_I(x) = T^-_t g_I(x) + T^+_t g_I(x) + T^{\mathrm {rem}}_t g_I(x),
\end{equation}
where
\begin{equation*}
T^\pm_t g_I(x)= |x|^{-\frac{d-1}{2}} \frac{e^{\mp i \frac \pi 4 (d-1)} }{(2\pi)^{d/2} }\int_0^\infty e^{i (t-t_I \pm |x|)s} \varphi(2^{-j}s) s^{\frac {d-1}{2}} {\text{\,\rm d}} s
\end{equation*}
and the remainder term is given by
\begin{equation*}
T^{\mathrm{rem}}_t g_I(x)= |x|^{-\frac{d-2}{2}} \int_0^\infty R(|x|s) \varphi(2^{-j}s) s^{ \frac d2} {\text{\,\rm d}} s.
\end{equation*}
Given $t \in E_j \cap I'$, let $J_t=[t-t_I-2^{-j-5}, t-t_I+2^{-j-5}]$, and define $D_t=\{x: |x| \in J_t\}$. We will examine each of the terms in \eqref{eq:lower bound 2} for $x\in D_t$.

If $x \in D_t$, then
\begin{equation*}
|T^-_t g_I(x)| \gtrsim |x|^{-\frac{d-1}{2}} \int_0^\infty \varphi(2^{-j}s) s^{\frac{d-1}{2}} {\text{\,\rm d}} s \gtrsim |x|^{-\frac{d-1}{2}} 2^{j \frac{d+1}{2}}.
\end{equation*}
Consequently,
\begin{align}
\| T^-_t g_I \|_{L^p(D_t)} \gtrsim 2^{j\frac{d+1}{2}} \Big(\int_{J_t} r^{ -(d-1)(\frac p2 -1)} {\text{\,\rm d}} r \Big)^{1/p} \gtrsim 2^{j(\frac{d+1}{2}-\frac 1p)} |I|^{-(d-1)(\frac{1}{2}-\frac{1}{p})} \label{eq:lower bound 3}
\end{align}
using that if $r \in J_t$, then $r \sim|t-t_I|\sim |I|$.

For the term $T^+_t$, we have by
repeated integration-by-parts,
\begin{align*}
|T^+_t g_I(x)| \lesssim_N |x|^{-\frac{(d-1)}{2}} 2^{j\frac{d-1}{2}} \frac{2^j}{\big(1+ 2^j ( |x| + t-t_I)\big)^N}
\end{align*}
for any $N>0$. Thus
\begin{align}
&\| T^+_t g_I \|_{L^p(D_t)} \lesssim 2^{j\frac{d+1}{2}} \Big( \int_{J_t} r^{-(d-1)(\frac{p}{2}-1)} \frac{1}{(1+ 2^j r)^N} {\text{\,\rm d}} r \Big)^{1/p} \notag\\
& \lesssim 2^{j \frac{d+1}{2}} |I|^{-(d-1)(\frac{1}{2}-\frac{1}{p})} (2^j |I|)^{-N/p} |J_t|^{\frac 1p} \le M^{-1} 2^{j(\frac{d+1}{2} -\frac 1p)} |I|^{-(d-1)(\frac{1}{2}-\frac{1}{p})}
\label{eq:lower bound 4}
\end{align}
using that $r \sim |I|$ for $r \in J_t$, $|J_t| \sim 2^{-j}$, $(2^j|I|)\ge M$ and choosing $N>p$.
For the remainder term we have for $x\in D_t$,
\begin{align*}
|T^{\mathrm{rem}}_t g_I(x)| & \lesssim |x|^{-\frac{(d+1)}{2}} \int_0^\infty \varphi(2^{-j}s) s^{\frac{s-3}{2}} {\text{\,\rm d}} s \lesssim |x|^{-\frac{(d+1)}{2}} 2^{j\frac{d+1}{2}} 2^{-j}.
\end{align*}
Then
\begin{align}
&\| T^{\mathrm{rem}}_t g_I \|_{L^p(D_t)} \lesssim 2^{j\frac{d-1}{2}} \Big(\int_{J_t} r^{-(d-1)(\frac{p}{2}-1)} r^{-p} {\text{\,\rm d}} r \Big)^{\frac 1p}\notag \\
& \lesssim 2^{j\frac{d-1}{2}} |I|^{-(d-1)(\frac{1}{2}-\frac{1}{p})} |I|^{-1} |J_t|^{\frac 1p} \lesssim M^{-1} 2^{j(\frac{d+1}{2} -\frac 1p)} |I|^{-(d-1)(\frac{1}{2}-\frac{1}{p})} \label{eq:lower bound 5}
\end{align}
since $r \sim |I|$ for $r \in I_t$ and $|I_t| \sim 2^{-j}$, $(2^j|I|)\ge M$.
From \eqref{eq:lower bound 4} and \eqref{eq:lower bound 5} we obtain, with a sufficiently large choice of $M$,
\begin{equation*}
\|T_t^+ g_I\|_{L^p(D_t)} + \|T^{\mathrm{rem}}_t g_I\|_{L^p(D_t)} \leq \frac{1}{2} \|T^-_t g_I \|_{L^p(D_t)}.
\end{equation*}
Combining this with \eqref{eq:lower bound 2} and the lower bound \eqref{eq:lower bound 3} for $T_t^-$, and taking the $\ell^p$ norm in $t\in E_j \cap I'$ we get
\begin{equation*}
\Big(\sum_{t \in E_j \cap I'} \|e^{i t\sqrt{-\Delta}} g_I\|_{L^p(D_t)}^p \Big)^{1/p} \gtrsim 2^{j (\frac{d+1}{2}-\frac 1p)} |I|^{-(d-1)(\frac{1}{2}-\frac 1p)} N( E \cap I, 2^{-j})^{\frac 1p};
\end{equation*}
here we used that $N(E \cap I', 2^{-j}) \geq \frac{1}{2} N(E \cap I, 2^{-j})$. Since $\| g_I \|_p \lesssim 2^{j(\frac{d+1}{2}-\frac 1p)}$ we obtain the desired lower bound \eqref{eq:lower bound 1}. \qed

\section{\texorpdfstring{Lower bounds for the $L^p\to \ell^q_{E_j} (L^q)$ conjecture}{Lower bounds for the conjecture}}
\label{sec:lowerboundsnonradial}
We construct counterexamples motivated by the examples for maximal operators in \cite{AHRS}, \cite{RoosSeeger}; these are associated to spherical pieces intermediate between spherical Knapp caps and full spheres. However, here we choose a sectorial localization on the Fourier side.
\begin{proposition}\label{prop:LpLqlower} Let $1\le p<\infty$, $q> p'$. Then there exist constants $c(q)$ and $\rho\ll 1$ such that for all intervals $I$ with $2^j\ge 2^j|I| \ge \rho^{-1} $
\begin{equation} \label{eq:LpLqlowerbd}
\sup_{\|f\|_p\le 1} \Big( \sum_{t\in E_j} \|e^{it\sqrt{-\Delta} } f \|_q^q \Big)^{1/q}
\ge c(q) \rho^{\frac dq} \frac{N(E\cap I, 2^{-j} )^{\frac 1q} 2^{j\frac{d+1}{2} (\frac 1p-\frac 1q)} }{|I|^{\frac{d-1}{2} (1-\frac 1p-\frac 1q)}}. \end{equation}
\end{proposition}
The case $p=2$ in the proposition shows that
in Theorem \ref{thm:sqfct} the critical $s_{E} (q)$ cannot be replaced by a smaller value for $r=q$, which implies the same conclusion for $2 \leq r \leq q$ by the nesting of the $\ell^r$ spaces. The case for general $p,q$ shows that in Conjecture \ref{conj:LpLq}
the exponent $s_E(p,q)$
cannot be replaced by a smaller one.

\begin{proof}[Proof of Proposition \ref{prop:LpLqlower}] Let $m \in \mathbb N$ be such that $2^j\le 2^m\le \rho^{-1}$ and let $I$ be an interval of length $2^{m-j} \le |I| \le 2^{m-j+1}$.
Let $I'$ be a subinterval of $I$ such that
$|I'|\approx \rho |I|$ and $N(E\cap I',2^{-j})\ge \rho N(E\cap I, 2^{-j})$. Let $t_I\in I' \cap E_j$.

Let $\upsilon$ be a nonnegative $C^\infty_c(B(0,1))$ function such that $\upsilon=1$ in a neighborhood of the origin.
Define $f_I\in L^p$ by its Fourier transform via
\begin{equation*} \widehat{f_I}(\xi)= (2\pi)^d \varphi(2^{-j}|\xi|) \upsilon( 2^{ \frac m2} (\tfrac \xi{|\xi|} -e_1)) e^{- i t_I|\xi|}.
\end{equation*} We first show that
\begin{equation} \label{eq:fJLpnorm} \|f_I\|_p\lesssim 2^{j(\frac{d+1}{2}-\frac 1p)} 2^{-m\frac{d-1}{2p}}.
\end{equation}
We
use a standard argument from \cite{SeegerSoggeStein} and decompose the Fourier transform into pieces supported on sectors of angular width $O(2^{-j/2})$.
For this decomposition split variables as $\xi=(\xi_1, \xi')$ and note that on the support of $\widehat f_I$ we have $\xi_1\approx 2^j$ and $|\xi'| \lesssim 2^{-m/2}\xi_1$.
Choose $\chi\in C^\infty_c({\mathbb {R}}^{d-1})$ supported in $(-1,1)^{d-1}$ such that $\sum_{{\mathfrak {z}}\in {\mathbb {Z}}^{d-1}} \chi( \xi'-{\mathfrak {z}}) =1$ for all $\xi'\in {\mathbb {R}}^{d-1}$. Then we write $f_I= \sum_{{\mathfrak {z}}}
f_{I,{\mathfrak {z}}} $ where
\[ \widehat {f_{I,{\mathfrak {z}}}} (\xi) = (2\pi)^d e^{-it_I|\xi|} \varphi(2^{-j}|\xi|) \upsilon( 2^{ \frac m2} (\tfrac \xi{|\xi|} -e_1)) \chi( 2^{j/2} \tfrac{\xi'}{\xi_1} - {\mathfrak {z}}) .\]
An integration-by-parts argument in \cite{SeegerSoggeStein} gives

\begin{equation} \label{eq:fzptw}|f_{I, {\mathfrak {z}}} (x)| \lesssim_N \frac{ 2^{j}} {(1+ 2^j |\inn{x}{e_{\mathfrak {z}}} -t_I|)^N} \frac{ 2^{j(d-1)/2 }}{(1+ 2^{j/2} |\pi_{\mathfrak {z}}^\perp x|)^N}
\end{equation}
where $e_{\mathfrak {z}}= \frac{ (1, 2^{-j/2} {\mathfrak {z}})}{\sqrt {1+2^{-j }|{\mathfrak {z}}|^2}}$
and $\pi_{\mathfrak {z}}^\perp$ is the orthogonal projection to the orthogonal complement of ${\mathbb {R}} e_{\mathfrak {z}}$. Here we use that $|\inn{e_{\mathfrak {z}}}{\nabla}^N \upsilon (2^{\frac{m}{2}}(\frac{\xi}{|\xi|} - e_1))| \lesssim_N 2^{-jN}$ for $|2^{-j/2} {\mathfrak {z}}| \lesssim 2^{-m/2}$ and also
$|\inn{e_{\mathfrak {z}}}{\nabla}^N \chi(2^{j/2}\frac{\xi'}{\xi_1}-{\mathfrak {z}}) |
\lesssim_N 2^{-jN}$ for all $N\ge 0$.

One computes that $\|f_{I,{\mathfrak {z}}}\|_1=O(1)$.
In view of the support properties of $\widehat f_I$ the sum $\sum_{\mathfrak {z}} f_{I,{\mathfrak {z}}} $ extends over $O(2^{\frac{j-m}{2}(d-1)})$ contributing terms and thus we get $\|f_I\|_1\lesssim 2^{ \frac{j-m}{2}(d-1)}$ which is
\eqref{eq:fJLpnorm} for $p=1$. Regarding $p=\infty$, we clearly have $\|f_{I,{\mathfrak {z}}}\|_\infty=O(2^{j\frac{d+1}{2}})$. However note that $t_I\approx 1$ and the vectors $t_Ie_{\mathfrak {z}}$ are $c2^{-j/2}$-separated, and thus one can use the decay properties in \eqref{eq:fzptw} to see that the same bound holds for the sum, $\sum_{\mathfrak {z}} f_{I,{\mathfrak {z}}}$. That is, we get $\|f_I\|_\infty\lesssim 2^{j\frac{d+1}{2}}$ which is \eqref{eq:fJLpnorm} for $p=\infty$. We now conclude \eqref{eq:fJLpnorm} using $\| f \|_p \leq \| f \|_1^{1/p} \| f \|_\infty^{1-1/p}$ for $1 \leq p \leq \infty$.

For $t \in I' \cap E_j$, let
\[R_{I,t} = \big\{ x=(x_1, x')\in {\mathbb {R}}^d: |x_1+t-t_I|\le 2^{-j},\,\, |x'|\le \rho 2^{-j+\frac m2} \big\}.\]
We will next prove a lower bound for $|e^{it\sqrt{-\Delta}} f_I(x)|$ and $x\in R_{I,t}$.

We use polar coordinates in the Fourier variable $\xi$ and write $\xi=r\theta(\omega)$ where $\omega\to \theta(\omega)$ is a smooth parametrization of $S^{d-1}$ near $e_1$ with $\theta(0)=e_1$. Here, the parameter $\omega$ lives in a neighborhood of the origin of ${\mathbb {R}}^{d-1}$.
Then
\begin{equation*} e^{ it\sqrt{-\Delta}} f_I(x)\\= \int \upsilon(2^{\frac m2} (\theta(\omega)-e_1))
\int r^{d-1} \varphi (2^{-j}r) e^{ i r (t-t_I + \inn{x}{\theta(\omega)}) }
{\text{\,\rm d}} r {\text{\,\rm d}}\sigma(\omega) .
\end{equation*}
We write
$ \inn{x}{\theta(\omega)}= x_1 + \inn{x} {\theta(\omega)-e_1} $ and
\[e^{it\sqrt{-\Delta}} f_I(x)= \mathrm I(x,t)+\mathrm{II}(x,t)\] where
\[ \mathrm I(x,t)=
\int r^{d-1} \varphi (2^{-j}r) e^{ ir(t-t_I+ x_1) } {\text{\,\rm d}} r
\int \upsilon(2^{\frac m2} (\theta(\omega)-e_1)) {\text{\,\rm d}}\sigma(\omega)
\] and
\begin{align*}
\mathrm{II}(x,t)
=\int r^{d-1} \varphi (2^{-j}r) e^{ ir(t-t_I+ x_1) }
\int u_m(r,\omega,x) {\text{\,\rm d}} r {\text{\,\rm d}}\sigma(\omega),
\end{align*}
with
\begin{equation*}
u_m(r,\omega, x)= \upsilon(2^{\frac m2} (\theta(\omega)-e_1)) \big( e^{ ir \inn{x}{\theta(\omega)-e_1}} -1\big).
\end{equation*}

For the term $\mathrm{I}(x,t)$ we set $\phi(r)=\varphi(r) r^{d-1}$ and note
\[\mathrm I(x,t) = c_m 2^{-m \frac{d-1}{2}} 2^{jd} \widehat\phi (2^j( t_I-t-x_1)) \]
with $c_m\approx 1$.
Using that $(\int_{-1/4}^{1/4} |\widehat{\phi}(s)|^q{\text{\,\rm d}} s)^{1/q}\gtrsim 1 $
we obtain the lower bound
\begin{equation} \label{eq:lowerboundILq} \Big(\int_{R_{I,t}} |\mathrm I(x,t)|^q {\text{\,\rm d}} x\Big)^{1/q} \gtrsim\rho^{\frac{d-1}{q}} 2^{jd( 1-\frac 1q)} 2^{-m \frac{d-1}{2} (1-\frac 1q) } . \end{equation}

For the term $\mathrm{II}(x,t)$ we get a corresponding upper bound, multiplied with an additional small factor of $\rho$. To see this, we expand
\begin{equation}\label{eq:n-expansion}
e^{ir \inn{x}{\theta(\omega)-e_1}} -1=\sum_{n=1}^\infty \frac{1}{n!}\, r^n (\inn{x}{\theta(\omega)-e_1})^n
\end{equation}
and write
\begin{equation*}
\inn{x}{\theta(\omega)-e_1} = x_1(\theta_1(\omega)-1) +\sum_{i=2}^{d} x_i\theta_i(\omega).
\end{equation*}
For $t\in I'$ and $x \in R_{I,t}$, we have $|x_1|\lesssim |t-t_I|\lesssim 2^{m-j} \rho $ and because of $\inn{e_1} {\frac{\partial \theta(\omega)}{\partial \omega_i}}|_{\omega=0}=0$ we get
$|x_1(\theta_1(\omega)-1) | \lesssim 2^{-j} \rho$.
Furthermore, for $i=2,\dots, d $, one has $|x_i\theta_i(\omega)|\lesssim 2^{-m/2}|x_i| \lesssim 2^{-j} \rho$. Thus,
\begin{equation}\label{eq:size expansion}
|\inn{x}{\theta(\omega)-e_1}| \lesssim 2^{-j}\rho \qquad \text{ for $x \in R_{I,t}, \,\, t \in I'$}.
\end{equation}
Using the expansion \eqref{eq:n-expansion} we write $\mathrm{II}=\sum_{n=1}^\infty \mathrm{II}_n$ and note the pointwise bounds
\[ n!|\mathrm{II}_n(x,t)| \le 2^{j(d+n) }|\widehat \phi_n(2^j(t_I-t-x_1))| \int |\upsilon(2^{\frac m2}(\theta(\omega)-e_1)| |\inn{x}{\theta(\omega)-e_1}|^n {\text{\,\rm d}}\sigma(\omega), \]
where $\phi_n(r) =\varphi (r) r^{d-1+n} $.
We have
$|\widehat \phi_n(y) | \le C_d n^{d+1} (1+|y|)^{-d-1} $ as can be seen using a $(d+1)$-fold integration-by-parts in $r$. Hence
\[ |\mathrm{II}_n(x,t)| \lesssim (C_d\rho)^n \frac{n^{d+1}}{n!} 2^{-m\frac{d-1}{2}} 2^{jd} (1+2^j|x_1 + t - t_I|)^{ -d-1} \quad \text{ for $x\in R_{I,t}$,} \]
using also \eqref{eq:size expansion}. Taking the $L^q(R_{I,t})$ norm and summing in $n\ge 1$ leads to
\begin{equation} \label{eq:upperboundIILq}
\Big(\int_{R_{I,t}} |\mathrm{II}(x,t)|^q {\text{\,\rm d}} x \Big)^{1/q} \lesssim_d \rho^{1+ \frac{d-1}{q} } 2^{jd(1-\frac{1}{q})} 2^{-m\frac{(d-1)}{2}(1-\frac{1}{q})} .
\end{equation}
Thus, for sufficiently small $\rho>0$ we get from \eqref{eq:fJLpnorm}, \eqref{eq:lowerboundILq}, and \eqref{eq:upperboundIILq}
\begin{equation*} \frac{\Big( {\sum_{t \in E_j \cap I'}} \| e^{i t \sqrt{-\Delta}} f_I \|_q^q \Big)^{1/q}} {\|f_I\|_p} \gtrsim \rho^{\frac{d-1}{q} } \frac{
N(E \cap I', 2^{-j})^{\frac 1q}
2^{jd(1-\frac 1q)} 2^{-m\frac{d-1}{2}( 1-\frac 1q) } }
{ 2^{j(\frac {d+1}2 -\frac 1p)} 2^{-m\frac{d-1}{2p}}}.
\end{equation*}
The right-hand side equals $\rho^{\frac{d-1}{q}} N(E\cap I', 2^{-j})^{\frac 1q} 2^{(m-j)\frac{d-1}{2}(\frac 1p+\frac 1q-1)} 2^{j\frac{d+1}{2} (\frac 1p-\frac 1q)}$. Since
$N(E \cap I', 2^{-j})\ge \rho N(E\cap I, 2^{-j})$, the lower bound
\eqref{eq:LpLqlowerbd} follows.
\end{proof}

\section{Upper bounds in Theorem \ref{thm:LSdiscr}}\label{sec:LS}

\begin{proposition} \label{prop:radialLpsm} Let $2 \leq p < \frac{2d}{d-1}$ and $s_p=(d-1)(\frac{1}{2}-\frac{1}{p})$. Then for $\varepsilon>0$ and sufficiently large $j \geq 1$,
\begin{equation}\label{eq:LS goal}
\Big( \sum_{t \in E_j} \|e^{i t \sqrt{-\Delta}}P_j f\|_{p}^p \Big)^{1/p} \lesssim_\varepsilon 2^{j\varepsilon} 2^{j \frac 1p\nu_E^\sharp(p s_p)} \| f \|_{L^p_{\text{\rm rad}}}.
\end{equation}
\end{proposition}
Once this is proven we can use that $\frac 1p\nu_E^\sharp(p s_p)= s_p$ for
$p>p_\gamma$ and obtain \eqref{eq:LS thm} by
interpolation (restricted to radial functions) of \eqref{eq:LS goal} with the $p=\infty$ version of the fixed-time estimate \eqref{eq:LS fixed time} for functions whose Fourier transform is supported in the annulus $\{\,|\xi|\approx 2^{-j}\}$.

\begin{proof}[Proof of Proposition \ref{prop:radialLpsm}]
For $0\le m\le j$ define
\begin{equation*}
\kappa_{j,m} =\sup_{|I|=2^{m-j}} N(E\cap I)|I|^{-p s_p}
\end{equation*}
and note that, by definition of $\nu_E^\sharp$, $ \kappa_{j,m} \le C_\varepsilon 2^{j\varepsilon} 2^{j \nu_E ^\sharp(p s_p)} $ for any $\varepsilon >0$.
It therefore suffices to prove that
\begin{equation}\label{eq:LS goalref}
\Big( \sum_{t \in E_j} \|e^{i t \sqrt{-\Delta}}P_j f\|_{p}^p \Big)^{1/p} \lesssim \Big(\sum_{m=0}^j \kappa_{j,m}\Big)^{1/p}
\| f \|_{L^p_{\text{\rm rad}}} , \quad 2\le p<\tfrac{2d}{d-1}\,.
\end{equation}

It was proven in \cite[Proposition 3.2]{MS-radial} that for $f$ radial, $|x|\geq 20$ and $t \in [1,2]$, the estimate
\[
\Big(\int_{|x| \geq 20} |e^{i t \sqrt{-\Delta}}P_j f(x)|^p {\text{\,\rm d}} x \Big)^{1/p} \lesssim \| f \|_{L^p_{\mathrm{rad}}}
\]
holds for all $2 \leq p < \infty$, with constant independent of $t$. Consequently,
\[
\Big(\sum_{t \in E_j} \big\|e^{i t \sqrt{-\Delta}}P_j f\big\|^p _{L^p(\mathbb R^d \backslash B(0,20))} \Big)^{1/p} \lesssim N(E, 2^{-j})^{1/p} \| f \|_{L^p_{\text{\rm rad}}}.
\]
Since $N(E, 2^{-j}) \leq \sup_{2^{-j}\le|I|\le 1} N(E \cap I, 2^{-j}) |I|^{-p s_p} \lesssim \max_{0\le m\le j}\kappa_{j,m} $, the inequality \eqref{eq:LS goalref} will follow from
\begin{equation}\label{eq:LS local}
\sum_{t \in E_j} \int_{|x|\le 20} |e^{i t \sqrt{-\Delta}}P_j f(x)|^p{\text{\,\rm d}} x \lesssim \sum_{j=0}^m\kappa_{j,m} \| f \|_{L^p_{\text{\rm rad}}}^p, \quad 2 \leq p < \tfrac{2d}{d-1}.
\end{equation}

Since $e^{i t|\xi|}$ is a radial Fourier multiplier for fixed $t \in [1,2]$, and $\varphi_j$ and $f$ are radial, we can write the operator $e^{it \sqrt{-\Delta}} P_j$ as
\begin{equation*}
e^{i t \sqrt{-\Delta}} P_j f(x)=\int_0^\infty K_j(|x|,t,s)f_0(s) {\text{\,\rm d}} s
\end{equation*}
where $f(x)=f_0(|x|)$ and
\begin{equation*}
K_j(r,s,t)=s^{d/2} r^{-(d-2)/2} \int_0^\infty J_{\frac{d-2}{2}}(\rho r) J_{\frac{d-2}{2}}(\rho s) \varphi (2^{-j} \rho) e^{i t \rho} {\text{\,\rm d}} \rho\,.
\end{equation*}
See, for instance, \cite[\S IV.3]{stein-weiss}. Here $J_{\frac{d-2}{2}}$ denotes the Bessel function of order $\frac{d-2}2$. Using asymptotics of Bessel functions and integration-by-parts,
it was shown in \cite[Lemma 2.1]{MS-radial} that these kernels satisfy the estimates
\begin{align*}
|K_j(r,s,t)|\lesssim \Big( \frac{s}{r} \Big)^{\frac{d-1}{2}} \sum_\pm \omega_j(t \pm r \pm s)
\end{align*}
where $\omega_j(u)\lesssim_N 2^j (1+2^j |u|)^{-N}$ for all $N>0$ and the sum is over all four choices of the two signs. Changing to polar coordinates and inserting the power weights into the function and operator, the inequality \eqref{eq:LS local} follows from the unweighted one-dimensional estimates
\begin{align}
\sum_{t \in E_j} \int_0^{20} & r^{(d-1)(1-\frac{p}{2})} \Big| \int_{0}^\infty s^{(d-1)(\frac{1}{2}-\frac{1}{p})} \omega_j(t \pm r \pm s) f_0(s) {\text{\,\rm d}} s \Big|^p {\text{\,\rm d}} r \notag \\
& \lesssim \sum_{m=0}^j \kappa_{j,m}
\int_0^\infty |f_0(s)|^p {\text{\,\rm d}} s \label{eq:LS reduced}
\end{align}
for all possible choices of $\pm$, where $2 \leq p < \frac{2d}{d-1}$. For $p \geq 2$, we dominate the left-hand side in \eqref{eq:LS reduced} by $\sum_{m=0}^{j} \mathrm I_m +\sum_{n \geq 10} \mathrm{II}_n$, where for $m<j$
\begin{align*}
\mathrm I_0& = \sum_{t \in E_j} \int_{0}^{2^{-j}} r^{(d-1)(1-\frac{p}{2})} \Big| \int_{0}^{2^{10}} s^{(d-1)(\frac{1}{2}-\frac{1}{p})} \omega_j(t \pm r \pm s) f_0(s) {\text{\,\rm d}} s \Big|^p {\text{\,\rm d}} r,\\
\mathrm I_m& = \sum_{t \in E_j} \int_{2^{-j+m}}^{2^{-j+m+1}} r^{(d-1)(1-\frac{p}{2})} \Big| \int_{0}^{2^{10}} s^{(d-1)(\frac{1}{2}-\frac{1}{p})} \omega_j(t \pm r \pm s) f_0(s) {\text{\,\rm d}} s \Big|^p {\text{\,\rm d}} r, \\
\mathrm I_j& = \sum_{t \in E_j} \int_{1}^{20} r^{(d-1)(1-\frac{p}{2})} \Big| \int_{0}^{2^{10}} s^{(d-1)(\frac{1}{2}-\frac{1}{p})} \omega_j(t \pm r \pm s) f_0(s) {\text{\,\rm d}} s \Big|^p {\text{\,\rm d}} r
\end{align*} and, with $n\ge 10$,
\begin{align*}
\mathrm{II}_n& = \sum_{t \in E_j} \int_0^{20} r^{(d-1)(1-\frac{p}{2})} \Big| \int_{2^n}^{2^{n+1}} s^{(d-1)(\frac{1}{2}-\frac{1}{p})} \omega_j(t \pm r \pm s) f_0(s) {\text{\,\rm d}} s \Big|^p {\text{\,\rm d}} r.
\end{align*}
The terms $\mathrm{II}_n$ can be treated in a straightforward manner. Note that, for $n \geq 10$, we have by Hölder's inequality
\begin{align*}
\mathrm{II}_n & \lesssim \sum_{t \in E_j} \int_0^{20} r^{(d-1)(1-\frac{p}{2})} \Big[ \int_{2^{n}}^{2^{n+1}} 2^{n(d-1)(\frac{1}{2}-\frac{1}{p})} 2^{-nN} 2^{-jN} |f_0(s)| {\text{\,\rm d}} s \Big]^p {\text{\,\rm d}} r \\
& \lesssim \sum_{t \in E_j} \Big( \int_0^{20} r^{(d-1)(1-\frac{p}{2})} {\text{\,\rm d}} r \Big) 2^{-jNp} 2^{-nN'p} \int_1^\infty |f_0(s)|^p \, {\text{\,\rm d}} s \\
& \lesssim N(E, 2^{-j}) 2^{-jNp} 2^{-nN'p} \int_1^\infty |f_0(s)|^p {\text{\,\rm d}} s
\end{align*}
for all $1 \leq p < \frac{2d}{d-1}$ and $N'>0$, provided $N$ is chosen sufficiently large. Then
\begin{equation}\label{eq:IIn terms}
\sum_{n \geq 10} \mathrm{II}_n \lesssim 2^{-jN} \| f_0 \|_p^p
\end{equation}
for all $1 \leq p < \frac{2d}{d-1}$ and any $N>0$.

We now turn to the terms $\mathrm I_m$. The term $\mathrm I_{j}$ is also trivial, since for $p \geq 2$
\begin{equation}\label{eq:Ij bound}
\mathrm I_j \lesssim \sum_{t \in E_j}\int_1^{20}\Big[ \int_0^{2^{10}} \omega_j(t \pm r \pm s) |f_0(s)| {\text{\,\rm d}} s \Big]^p {\text{\,\rm d}} r \lesssim N(E, 2^{-j}) \| f_0 \|_p^p
\end{equation}
by Young's convolution inequality, noting that $\|\omega_j \|_1 \lesssim 1$. For the term $\mathrm I_0$, we define for each $t \in [1,2]$ the interval $J_{t,j}=[t-2^{-j+3}, t+2^{-j+3}]$. We note that for $t \in [1,2]$ and $0 \leq r \leq 2^{-j}$
\begin{equation*}
\int_0^{2^{10}} \omega_j(t \pm r \pm s) |f_0(s)| {\mathbbm 1}_{J_{t,j}^\complement}(s) {\text{\,\rm d}} s \lesssim 2^{-jN} \int_0^{2^{10}} |f_0(s)|{\text{\,\rm d}} s \lesssim 2^{-jN} \|f_0 \|_p
\end{equation*}
for any $N>0$, and
\begin{equation*}
\int_0^{2^{10}} \omega_j(t \pm r \pm s) |f_0(s)| {\mathbbm 1}_{J_{t,j}}(s) {\text{\,\rm d}} s \lesssim \| \omega_j \|_{p'} \| f_0 {\mathbbm 1}_{J_{t,j}} \|_p \lesssim 2^{j/p} \| f_0 {\mathbbm 1}_{J_{t,j}} \|_p,
\end{equation*}
where both inequalities follow from the bound $\omega_j(u)\lesssim_N 2^{j}(1+2^j|u|)^{-N}$ for any $N>0$ and Hölder's inequality. Using these,
\begin{align}
\mathrm I_0 & \lesssim \Big(\int_0^{2^{-j}} r^{(d-1)(1-\frac{p}{2})} {\text{\,\rm d}} r \Big) \Big( 2^{-jNp}N(E,2^{-j})\| f_0 \|_p^p + 2^j \sum_{t \in E_j} \| f_0 {\mathbbm 1}_{I_{t,j}} \|_p^p \Big) \notag \\&\lesssim 2^{j(d-1)(\frac{p}{2}-1)}\| f_0 \|_p^p \lesssim 2^{j \nu^\sharp_E(p s_p)} \|f_0 \|_p^p\label{eq:I0 bound}
\end{align}
for $2 \leq p < \frac{2d}{d-1}$, using that $(d-1)(\frac{p}{2}-1)=p s_p \leq \nu^\sharp_E(p s_p)$.

We now address the main terms with $0 < m < j$. We decompose $[1,2]$ into disjoint intervals $\{ I_\mu\}$ of length $|I_\mu|=2^{-j+m}$, and denote by $I_\mu^*$ the concentric interval with 5 times the length.
Then
\begin{align*}
\mathrm I_m & \lesssim 2^{(m-j)(d-1)(1-\frac{p}{2})} \sum_\mu \sum_{t \in E_j \cap I_\mu} \int_{2^{-j+m}}^{2^{-j+m+1}} \Big[ \int_0^{2^{10}} \omega_j(t \pm r \pm s) |f_0(s)| {\text{\,\rm d}} s\Big]^p {\text{\,\rm d}} r \notag \\
& \lesssim
\sup_{|I|=2^{m-j}}|I|^{-(d-1)(\frac{p}{2}-1)} \#(E_j \cap I) \sum_\mu \int_{I_\mu^*} \Big[ \int_0^{2^{10}} \omega_j(r' \pm s) |f_0(s) |{\text{\,\rm d}} s\Big]^p {\text{\,\rm d}} r' \notag \\
& \lesssim \kappa_{j,m}
\int_{1}^2 \Big[ \int_0^{2^{10}} \omega_j(r' \pm s) |f_0(s)| {\text{\,\rm d}} s\Big]^p {\text{\,\rm d}} r' \lesssim \kappa_{j,m} \int_0^{2^{10}} |f_0(s)|^p {\text{\,\rm d}} s,
\end{align*}
by the change of variables $r'=t \pm r$ and noting that $\| \omega_j \|_1 \lesssim 1$.
Combining this estimate with \eqref{eq:Ij bound} and \eqref{eq:I0 bound} we get
\begin{equation}\label{eq:Im terms}
\sum_{m=0}^{j} \mathrm I_m \lesssim \sum_{m=0}^j\kappa_{j,m} \|f\|_p^p\,.
\end{equation}
By \eqref{eq:Im terms} and \eqref{eq:IIn terms} we obtain \eqref{eq:LS reduced} for all $2 \leq p < \frac{2d}{d-1}$, which concludes the proof.
\end{proof}

\section{Upper bounds in Theorem \ref{thm:sqfct}}\label{sec:SF}

For the upper bounds in Theorem \ref{thm:sqfct} it suffices to settle the case $r=2$.
Setting $T^j_t=e^{it\sqrt{-\Delta}}P_j$
we get from Young's inequality and Plancherel's theorem
\begin{equation}\label{eq:fixedtimeSobineq}\|T_t^j\|_{L^2\to L^q}\le 2^{jd (1/2-1/q)}, \quad 2\le q\le\infty. \end{equation}
Moreover, by the usual $TT^*$ argument \cite{Strichartz}, the asserted $L^2\to L^q(\ell^2_{E_j})$
bound for $\{T^j_t\}_{t\in E_j}$ is equivalent with the inequality
\begin{equation}\label{eq:TT*}
\|S^jg\|_{L^q(\ell^2_{E_j})} \lesssim
2^{2js}\|g\|_{L^{q'}(\ell^2_{E_j})}
\end{equation}
where
\begin{equation*}
S^j g(y,t) = \sum_{t'\in E_j} T_t^j (T_{t'}^j)^* [g(\cdot,t')](y)
\end{equation*}
and $s > s_E(q)$.
The Schwartz kernel of $T_t^j (T_{t'}^j)^*$ is given by $2^{jd} {\mathcal {K}}_j(y,t,y',t')$ where
\begin{equation*}
{\mathcal {K}}_j(y,t,y',t')=\frac{1}{(2\pi)^d}\int_{\mathbb R^d} |\varphi(|\xi|)|^2
e^{i 2^j [(t-t')|\xi|+\inn{y-y'}{\xi}]} {\text{\,\rm d}} \xi\,.
\end{equation*}
Let $\tilde{\eta}\in C^\infty_c({\mathbb {R}}^d)$ with $\tilde{\eta}(w)=1$ for $|w|\le 1/2$ and $\tilde{\eta}$ compactly supported in $\{|w|\le 1\}$. Let $\eta(w)=\tilde{\eta}(w)-\tilde{\eta}(2w)$. Now we set $\tilde{\eta}_{-j}(w)=\tilde{\eta} (2^j w)$ and $\eta_{k}(w)= \eta(2^{-k} w)$, so that $1=\tilde{\eta}_{-j}+\sum_{m\ge 1} \eta_{m-j}$ for every $j$.
We decompose
\begin{equation}\label{eq:dec S^j}
S^j = S^j_0+\sum_{m>0} S^j_m + \sum_{m>0} R^j_m
\end{equation}
with
\[ S_0^j g(y,t) = 2^{jd} \sum_{t'\in E_j}
\int_{{\mathbb {R}}^d} {\mathcal {K}}_j(y,t,y',t') \tilde{\eta}_{-j}(y-y') g(y',t') {\text{\,\rm d}} y'
\]
and, for $m\ge 1$
\begin{align*} S_m^j g(y,t)
&= 2^{jd} \sum_{\substack {t'\in E_j \\ |t-t'|\le 2^{m-j+10} }}
\int_{{\mathbb {R}}^d} {\mathcal {K}}_j(y,t,y',t') \eta_{m-j}(y-y') g(y',t') {\text{\,\rm d}} y',
\\
R_m^j g(y,t) &= 2^{jd} \sum_{\substack {t'\in E_j \\ |t-t'|> 2^{m-j+10} }} \int_{\mathbb R^d} {\mathcal {K}}_j(y,t,y',t') \eta_{m-j}(y-y')
g(y',t') {\text{\,\rm d}} y'.
\end{align*}
The term $S_0^j$ is trivial, and the terms $R_m^j$ and $S^{j}_m$ with $m> j+10$ can be seen as error terms.
\begin{lemma}\label{sqfct-smallbds}
Let $2 \leq q \leq \infty$. For all $j\ge 0$, $N\ge 0$, we have the following bounds.
\begin{enumerate}[(i)]
\item $\|S_0^j\|_{L^{q'}(\ell^2_{E_j}) \to L^q(\ell^2_{E_j})} \lesssim 2^{jd(1-2/q)}$.
\item
For $m>j+10$, $\|S^j_m\|_{L^{q'}(\ell^2_{E_j}) \to L^q(\ell^2_{E_j})} \lesssim_N 2^{-(j+m)N}$.
\item
For $m > 0$, $\|R_m^j\|_{L^{q'}(\ell^2_{E_j}) \to L^q(\ell^2_{E_j})} \lesssim_N 2^{-(j+m)N}$.
\end{enumerate}
\end{lemma}

Part (i) follows from \eqref{eq:fixedtimeSobineq} and the Cauchy--Schwarz inequality. The proof of (ii) and (iii) is straightforward, and based on
\begin{equation*}
|{\mathcal {K}}_j(y,t,y',t') |\lesssim_{M}
\begin{cases}
(1+2^j|t-t'|)^{-M}
&\text{ if $|t-t'| \ge 2|y-y'|$ }
\\
(1+2^j|y-y'|)^{-M} &\text{ if $|y-y'|\ge 2|t-t'|$}
\end{cases}
\end{equation*}
for any $M>0$, which is obtained using integration-by-parts. We omit the details.

The main contribution comes from the terms $S^j_m$ with $m\le j+10$.

\begin{lemma} \label{lem:lambdaq}
Let $2 \leq q \leq \infty$ and
\begin{equation*}
\lambda_{j,m}= 2^{jd(1-\frac 2q)} 2^{-m(d-1)(\frac 12-\frac 1q)} \sup_{|I|=2^{m-j}} N(E\cap I, 2^{-j}) ^{\frac 2q}.
\end{equation*}
For $m \leq j + 10$,
\begin{equation}\label{eq:laq-bound}
\|S^j_m g\|_{L^{q} (\ell^2_{E_j})} \lesssim \lambda_{j,m} \|g\|_{L^{q'}(\ell^2_{E_j})}.
\end{equation}
\end{lemma}
\begin{proof}
By interpolation, it suffices to show \eqref{eq:laq-bound} for $q=2$ and $q=\infty$.

\subsubsection*{Case $q=\infty$}
After changing to polar coordinates, we express $(2\pi)^d 2^{-jd} S^j_m g(y,t)$ as
\begin{equation}\label{eq:tspf1}
\sum_{t'\in E_j} \int_{\mathbb R^d} \int_{0}^\infty \int_{S^{d-1}} |\varphi(r)|^2 r^{d-1} e^{i 2^j r [(t-t')+\langle y-y',\theta\rangle]}
g(y',t') \eta_{m-j}(y-y') {\text{\,\rm d}}\sigma(\theta) {\text{\,\rm d}} r {\text{\,\rm d}} y'
\end{equation}
with $\mathrm{d}\sigma$ denoting the normalized surface measure on $S^{d-1}$.
It is well-known \cite{Stein1993} that
\begin{equation*}
\int_{S^{d-1}} e^{i \langle y,\theta\rangle}
{\text{\,\rm d}}\sigma(\theta) = \sum_{\pm} e^{\pm i |y|} b_\pm (y)
\end{equation*}
for smooth symbols $b_\pm$ satisfying $\partial^\alpha b_\pm (w) \lesssim_{\alpha} |w|^{-\frac{d-1}{2}-|\alpha|}$ for $|w|\ge 1$ and $\alpha \in \mathbb N_0^d$. Then \eqref{eq:tspf1} becomes
\begin{equation*}
\sum_{\pm}\sum_{t'\in E_j} \int_{\mathbb R^d} \int_{0}^\infty b_{\pm}(2^j r (y-y')) |\varphi(r)|^2 r^{d-1} e^{i 2^j r [(t-t')\pm |y-y'|]}
g(y',t') \eta_{m-j}(y-y') {\text{\,\rm d}} r{\text{\,\rm d}} y'.
\end{equation*}
Integrating by parts in $r$, we estimate
\begin{equation*} 2^{-jd}|S_m^j g(y,t)| \lesssim 2^{-m\frac{d-1}2} \sum_{\pm}\sum_{t'\in E_j} \int_{\mathbb R^d} \big(1+2^j \big|t-t'\pm |y-y'|\big|\big)^{-N}
|g(y',t')| {\text{\,\rm d}} y'
\end{equation*}
for any $N>0$. For fixed $y$, the $\ell^2(E_j)$ norm in $t$ of this expression
is bounded by
\begin{align*}
&2^{-m\frac{d-1}{2} } \sum_\pm \int_{{\mathbb {R}}^d} \Big( \sum_{t\in E_j} \Big| \sum_{t'\in E_j}
\big(1+2^j\big|t-t'\pm |y-y'|\big|\big)^{-N} g(y', t')\Big|^2\Big)^{1/2} {\text{\,\rm d}} y'
\\
&\lesssim 2^{-m\frac{d-1}{2} } \sum_\pm \int_{{\mathbb {R}}^d} \Big( \sum_{t'\in E_j}|g(y', t')|^2\Big)^{1/2} {\text{\,\rm d}} y'
\end{align*}
where we applied Schur's test on the $1$-separated set $\{2^j t:t\in E_j\}$. Combining the above we get
\begin{equation*}
|S_m^j g(y,\cdot)|_{\ell^2_{E_j} } \lesssim 2^{jd} 2^{-m\frac{d-1}{2} } \|g\|_{L^1(\ell^2_{E_j})}
\end{equation*}
which yields \eqref{eq:laq-bound} with $q=\infty$.

\subsubsection*{Case $q=2$}

Using the Fourier inversion theorem for $\eta_{m-j}(y-y')$ we write
\begin{equation*}
S^j_m g (y,t) =
\frac{1}{(2\pi)^d} \int_{\mathbb R^d} \widehat \eta (\omega) e^{ i 2^{j-m}\inn{y}{\omega }} \!\!\! \sum_{\substack{ t'\in E_j\\ |t-t'|\le 2^{m-j+10} } } \!\!\! T_t^j (T_{t'}^j)^* \big[ g(\cdot, t') e^{- i 2^{j-m} \inn{y'}{\omega}} \big]{\text{\,\rm d}} \omega.
\end{equation*}
In view of the rapid decay of $\widehat \eta(\omega)$, the inequality \eqref{eq:laq-bound} for $q=2$ follows via Minkowski's inequality from
\begin{equation} \label{eq:TT*L2}\Big\| \Big(\sum_{t'\in E_j} \Big|
\sum_{\substack{ t\in E_j\\ |t-t'|\le 2^{m-j+10} } }T_t^j (T_{t'}^j)^* [ g(\cdot, t')] \Big|^2\Big)^{1/2}\Big\|_2
\lesssim \sup_{|I| =2^{m-j}} N(E\cap I, 2^{-j} ) \|g\|_{L^2(\ell^2_{E_j})}.
\end{equation}
For $\mu \in {\mathbb {Z}}$ we let $I_\mu=[\mu 2^{m-j}, (\mu+1) 2^{m-j}]$. The left-hand side above is equal to
\begin{align*}
& \Big (\sum_\mu \sum_{t\in E_j\cap I_\mu } \int_{\mathbb R^d} \Big| \sum_{\substack{t'\in E_j\\|t-t'|\le 2^{m-j+10}}} T_t^j (T_{t'}^j)^* [ g(\cdot, t')] (y )\Big |^2 {\text{\,\rm d}} y\Big)^{\frac 12}
\\
&\lesssim \Big ( \sum_{\substack{(\mu,\mu')\\|\mu-\mu'|\le 2^{11}}} N(E_j\cap I_{\mu'} , 2^{-j})\sum_{t\in E_j\cap I_\mu } \sum_{ t'\in E_j\cap I_{\mu'}} \int|
T_t^j (T_{t'}^j)^* [ g(\cdot, t')] (y )|^2
{\text{\,\rm d}} y \Big)^{\frac12}
\\
& \lesssim \sup_{|I|=2^{m-j}} N(E\cap I, 2^{-j})^{\frac 12} \Big (\sum_{t'\in E_j} \sum_{\substack{t\in E_j\\|t-t'|\lesssim 2^{m-j}}} \|g(\cdot,t')\|_2^2\Big)^{1/2} \\
& \lesssim \sup_{|I|=2^{m-j}} N(E\cap I, 2^{-j})
\|g\|_{L^2(\ell^2_{E_j})}
\end{align*}
where we have used Cauchy--Schwarz in the first inequality and $\|T_t^j\|_{2 \to 2}=1$ in the second inequality. Thus \eqref{eq:TT*L2} follows. This finishes the proof of Lemma \ref{lem:lambdaq}.
\end{proof}

\begin{proof}[Proof of Theorem \ref{thm:sqfct} (upper bounds)]
As discussed at the beginning of this section, the inequality \eqref{eq:Lqell2} follows from \eqref{eq:TT*}. To prove the latter, we use the decomposition \eqref{eq:dec S^j} and the triangle inequality.
By Lemma \ref{sqfct-smallbds}
\begin{equation*}
\|S^j_0g\|_{L^{q}(\ell^2_{E_j})}+
\sum_{m\ge j+10} \|S^j_m g\|_{L^{q}(\ell^2_{E_j})}
+
\sum_{m\ge 0} \|R^j_m g\|_{L^{q}(\ell^2_{E_j})}
\lesssim 2^{jd(1-\frac 2q)} \| g\|_{L^{q'}(\ell^2_{E_j})}.
\end{equation*}
Since $\nu^\sharp_E(\alpha)\ge \alpha$, we have
$2^{j d(1-\frac{2}{q})} \leq 2^{j(d+1)(\frac{1}{2}-\frac{1}{q})}2^{\frac{2j}{q} \nu^\sharp_E(\frac{d-1}{2}(\frac{q}{2}-1))}=2^{2js_E(q)}$
and thus the above bound is admissible towards proving \eqref{eq:TT*}.

We next turn to the terms $S_m^j$ with $m<j+10$. By the definitions of $\lambda_{j,m}$, $\nu^\sharp_E$ and $s_{E}(q)$ we get for $\varepsilon>0$
\begin{align*}
\lambda_{j,m} &\le 2^{j(d+1)(\frac 12-\frac 1q)} \big[ \sup_{|I|=2^{m-j}} |I|^{- (d-1)(\frac q2-1)} N(E\cap I, 2^{-j}) \big]^{\frac 2q} \\&\lesssim_{\varepsilon} 2^{j(d+1)(\frac 12-\frac 1q)} 2^{\frac {2j} q (\nu^\sharp_E( \frac{d-1}{2}(\frac q2-1)) +\varepsilon ) }\leq 2^{j(2s_E(q) + \varepsilon)}.
\end{align*}
Thus, by Lemma \ref{lem:lambdaq} we obtain for
$s > s_E(q)$
\begin{align*} \sum_{m\le j+10} \|S^j_m g\|_{L^{q}(\ell^2_{E_j})}
&\lesssim_{\varepsilon} (1+j) 2^{j(2s_E(q)+\varepsilon)}
\| g\|_{L^{q'}(\ell^2_{E_j})}
\lesssim_s 2^{2js}
\| g\|_{L^{q'}(\ell^2_{E_j})}.
\end{align*}
which concludes the proof of \eqref{eq:TT*}.
\end{proof}
The above argument can also recover a sharper $L^2\to L^q$ result for the square-function associated with a set of given Assouad dimension that was previously proved independently by Wheeler \cite{Wheeler2024}, and by the first, second and fourth author (unpublished). Recall that $\dim_{\mathrm A} E$ is the infimum over all exponents $a$ such that $N(E\cap I,\delta) \le C_{a} (|I|/\delta)^a$ holds for all $\delta\in (0,1)$ and all intervals $I$ with $\delta\le |I|\le 1$.

\begin{proposition}\label{prop:Str-Assouad}
Let $2\le r \leq q<\infty$,
$\gamma_\circ=\dim_{\mathrm A} E$ and
$q_\circ=\frac{2(d-1+2\gamma_\circ)}{d-1}. $

(i) For $q>q_\circ$,
\begin{equation*}
\Big\| \Big(\sum_{t\in E_j} |e^{-it\sqrt{-\Delta}} P_j f|^r\Big)^{1/r} \Big\|_q \lesssim 2^{j d(\frac 12-\frac 1q)} \|f\|_2.
\end{equation*}

(ii) Suppose $\sup_{0<\delta<|I|\le 1} \big(\frac{\delta}{|I|} \big)^{\gamma_\circ} N(E\cap I, \delta) <\infty$. Then
\begin{equation*}
\Big\| \Big(\sum_{t\in E_j} |e^{-it\sqrt{-\Delta}} P_j f|^r\Big)^{1/r} \Big\|_{L^{q_\circ,\infty}} \lesssim 2^{j d(\frac 12-\frac 1{q_\circ})} \|f\|_2.
\end{equation*}
\end{proposition}
\begin{proof}
It suffices to show the $r=2$ case.

For part (i) we have $\lambda_{j,m}\lesssim_\varepsilon 2^{jd(1-\frac 2q)} 2^{-m((d-1)(\frac 12-\frac 1q) - \frac{2\gamma_\circ}{q}+\varepsilon)}$
and thus \begin{equation} \label{eq:Sjmbd}
\|S^j_m g\|_{L^{q}(\ell^2_{E_j})}
\lesssim_\varepsilon 2^{jd(1-\frac 2q)} 2^{-m((d-1)(\frac 12-\frac 1q) + \frac{2\gamma_\circ}{q}+\varepsilon)}\| g\|_{L^{q'}(\ell^2_{E_j})}.
\end{equation}
Observe that
$(d-1)(\frac 12-\frac 1q) - \frac{2\gamma_\circ}{q}>0$ for $q>q_\circ$, hence we may sum in $m\le j+10$ in this range and get part (i) by the usual $TT^*$ argument used above.

For part (ii), we use the stronger assumption $N(E\cap I,\delta)\lesssim (\delta/|I|)^{-\gamma_\circ}$ to get \eqref{eq:Sjmbd} with $\varepsilon=0$.
Then Bourgain's interpolation trick \cite{Bourgain-CompteRendu1985, CSWW1999} yields
\begin{equation*}
\Big\|\sum_{m\le j+10} S^j_m g\Big\|_{L^{q_\circ,\infty}(\ell^2_{E_j})}
\lesssim 2^{jd(1-\frac 2{q_\circ})} \| g\|_{L^{q_\circ',1} (\ell^2_{E_j})}.
\end{equation*}
Now part (ii) follows again by a $TT^*$ argument, using duality for vector-valued Lorentz spaces.
\end{proof}

\end{document}